\documentclass[10pt]{article}

\usepackage{tmlr}

\usepackage{amsmath, amssymb, amsthm}
\usepackage{physics}
\usepackage{graphicx}
\usepackage{subcaption}
\usepackage{xcolor}
\usepackage{tikz}
\usepackage{hyperref}
\usepackage{algorithm, algorithmic}
\usepackage{array}
\usepackage{lmodern}
\usepackage{geometry}
\usepackage{indentfirst}
\usepackage{titlesec}
\usepackage{fancyhdr}
\usepackage{appendix}
\usepackage{tikz-cd}
\newtheorem{theorem}{Theorem}[section]
\newtheorem{lemma}[theorem]{Lemma}
\newtheorem{definition}[theorem]{Definition}
\newtheorem{proposition}[theorem]{Proposition}
\newtheorem{corollary}[theorem]{Corollary}

\newtheorem{remark}[theorem]{Remark}
\titleformat{\section}
  {\Large\bfseries}
  {\thesection}{0.5em}{}

\hypersetup{
    colorlinks = true,
    linkcolor  = black,
    urlcolor   = black,
    citecolor  = black
}

\renewenvironment{proof}[1][Proof]{%
  \noindent\textit{#1. }%
}{\hfill$\square$\vspace{4pt}}

\begin{document}
\title{Cauchy problem for a Schrödinger-type equation related to the Riemann zeta function} 
	\author{Mohamed Bensaid\\Univ. Lille, CNRS, Inria, UMR 8524 - Laboratoire Paul Painlevé, F-59000 Lille, France\\ \url{mohamed.bensaid@univ-lille.fr}}
\maketitle
\begin{abstract}
We study the Cauchy problem in the space $H^1(\Sigma)$ for a nonlinear damped Schrödinger equation of the form
\begin{equation}\tag{NLS-$\zeta$}\label{nls}
    i u_t + \Delta u + i \lambda u \, \zeta(|u|+1) = 0, \quad u(0,x) = u_0,
\end{equation}
where $\zeta$ denotes the Riemann zeta function. We first establish the uniqueness of solutions in the sense of distributions. Then, by considering a regularized problem, we prove the existence of a global solution in $H^1(\Sigma)$, using uniform estimates and compactness arguments. Finally, we show that the limiting solution indeed satisfies the original equation in the weak sense. In the addition we proof that, the one-dimensional case, we show that it becomes zero in finite time.
\end{abstract}

\section{Introduction}
We consider the Schr\"odinger equation with a homogeneous damping term related to the famous Riemann zeta function:
\begin{equation}\tag{NLS-$\zeta$}\label{nlss}
    i u_t + \Delta u + i \lambda u \, \zeta(|u|+1) = 0, \quad u(0,x) = u_0,
\end{equation}
where the Riemann zeta function is initially defined for $\Re(s) > 1$ by
\[
\zeta(s) = \sum_{n=1}^\infty \frac{1}{n^s}.
\]

Furthermore, we consider a smooth, compact, finite-dimensional Riemannian manifold $\Sigma$ without boundary. On $\mathbb{R}^d$, typically only on a compact manifold does 
$u \, \zeta(|u|+1)$ belong to $L^p_x(\mathbb{R}^d)$ for finite $p$, 
so the nonlinear term is delicate to control in the $\mathbb{R}^d$ case.

As $x \to 0$, we know the asymptotic behavior
\[
\zeta(x+1) \sim \frac{1}{x}.
\]

In this context, the Cauchy theory in $H^{1}(\Sigma)$ for the equation
\begin{equation}\tag{E0}\label{E0}
i v_t + \Delta v + i\lambda \frac{v}{|v|} = 0, 
\qquad v(0,x) = v_0,
\end{equation}
was established in \cite{3} by R\'emi Carles and Cl\'ement Gallo. They also proved that the solution becomes zero for $d \in \{1,2,3\}$ under suitable regularity conditions.

Inspired by \cite{3} and \cite{1}, we establish the Cauchy theory for the $(\text{NLS-}\zeta)$ equation. In addition, we provide an explicit estimate for the mass of solutions, and we show that finite-time extinction occurs under the conditions for the dimension $d$.

In the final section, we also study the effect of combining the nonlinearity with a logarithmic term.

\subsection{Main result}
Before stating our main result, we need to specify the notion of weak solution. Recall that for $\Re(s) > 0$, we have
\[
\zeta(s+1) = \frac{1}{s} + \psi(s+1),
\]
for some holomorphic function $s \mapsto \psi(s)$. Hence, in \eqref{nls}, the nonlinearity has the form
\[
i \lambda u \, \zeta(|u| + 1) = i\lambda\big(\frac{u}{|u|} + u \, \psi(|u| + 1)\big).
\]
However, this expression is not defined when $u = 0$. Therefore, it is impossible to consider $\frac{u}{|u|}$ as a function everywhere.

\begin{definition}
    We say that $u$ is weak solution to \eqref{nls} if
\[
u \in \mathcal{C}(\mathbb{R}_+; L^2(\Sigma)) \cap L^{\infty}(\mathbb{R}_+; H^1(\Sigma))
\]
satisfying
\[
i \frac{\partial u}{\partial t} + \Delta u + i \lambda F=0
\quad \text{in } \mathcal{D}'(\mathbb{R}_+^* \times \Sigma),
\]
where the function $F$ satisfies
\[
\|F\|_{L^{\infty}(\mathbb{R}_+ \times \Sigma)} \le C, \quad\text{ and }\quad F=\dfrac{u}{|u|} \quad\text{if } u \neq 0.
\]
for some $C>0$
\end{definition}
\begin{theorem}\label{1.1}
    Let $d\in \mathbb{N}^*$, for $\lambda>0$ 
for any initial data $u_0 \in H^1(\Sigma)$, there exists a unique global weak solution
\[
u \in L^{\infty}(\mathbb{R}_+,H^1(\Sigma))\cap \mathcal{C}(\mathbb{R}_+,L^2(\Sigma)).
\]
for \eqref{nls}.
In the addition $$\norm{u(t)}_2\leq \norm{u_0}_2e^{-\lambda t} \quad \text{ and }\quad \norm{\nabla u(t)}_2\leq \norm{\nabla u_0}_2.$$
Morevore if $d=1$, then there exists $T > 0$ such that the (unique) weak
solution to \eqref{nls} satisfies
\[
u(t,x) = 0, \quad \text{for almost every } x \in \Sigma, \quad \text{for every } t > T.
\]
\end{theorem}
\begin{proposition}\label{pro}
   Let $u,v \in L^\infty(\mathbb{R}_+;H^1(\Sigma))$ be two (distributional) solutions of \eqref{nls} with initial data $u_0,v_0$ given by previos theorem. then for all $t\geq s$, $$\norm{u(t)-v(t)}_2\leq e^{-c(t-s)}\norm{u(s)-v(s)}_2.$$
   for some $c:=c(\lambda)>0$.
\end{proposition}
\begin{corollary}[Continuity of the flow]\label{flow}
     Let $T>0$ and $u_{n,0},\, u_0 \in H^1(\Sigma)$, and let $u_n(t)$ and $u(t)$ be the solutions corresponding to the initial data $u_{n,0}$ and $u_0$, respectively for \eqref{nls}. Suppose that
\[
u_{n,0} \longrightarrow u_0 \quad \text{in } L^2(\mathbb{R}^d),
\] 
and that there exists a constant $K>0$ such that
\[
\sup_{t \in [0,T]} \|u_n(t)\|_{H^1} \le K.
\] 
Then, for all $t \in [0,T]$, we have (up to subsequence)
\[
u_n(t) \rightharpoonup u(t) \quad \text{ weakly in } H^1(\Sigma) \qquad \text{ and }\qquad u_n(t) \longrightarrow u(t) \quad \text{ strongly in } L^2(\Sigma)
\]
\end{corollary}
The extension of the Finite time extinction property of solutions to higher dimensions 
remains an open and challenging question. 
The main difficulty lies in controlling the relevant norms and obtaining suitable inequalities 
in higher-dimensional spaces, which are essential for applying the extinction argument. 
In \cite{3}, additional regularity is used to deduce finite-time extinction for $d=2,3$. 
This idea could also be extended to equation~\eqref{nls},
however, in the present case, we know more precisely that the mass decreases exponentially, 
in contrast to the situation studied in \cite{3}.

\begin{lemma}\label{coro}
        Let $d=1$, for $\lambda>0$ 
for any initial data $u_0 \in H^1(\Sigma)$, there exists a unique global solution
\[
u \in L^{\infty}(\mathbb{R}_+,H^1(\Sigma))\cap \mathcal{C}(\mathbb{R}_+,L^2(\Sigma)).
\]
satisfait \begin{equation}\tag{logNLS-$\zeta$}\label{Elog}
\begin{cases}
i u_t + \Delta u + i\lambda\, u\,\zeta(|u|+1) + \mu\, u\,\log(|u|) = 0, & (t,x)\in \mathbb{R}_+ \times \Sigma,\\[0.3em]
u(0,x) = u_0(x)\in H^1(\Sigma),
\end{cases}
\end{equation}
In the addition $$\norm{u(t)}_2\leq \norm{u_0}_2e^{-\lambda t} \quad \text{ and }\quad \norm{\nabla u(t)}_2\leq \norm{\nabla u_0}_2.$$
Morevore, there exists $T > 0$ such that the (unique) weak
solution to \eqref{Elog} satisfies
\[
u(t,x) = 0, \quad \text{for almost every } x \in \Sigma, \quad \text{for every } t > T.
\]

\end{lemma}

\begin{remark}
\begin{enumerate}
    \item
    The Cauchy problem can also be extended to higher spatial dimensions $ d $.
    However, the solutions obtained belong to
    $
    L^\infty_{\mathrm{loc}}(\mathbb{R}_+; H^1(\Sigma))
    \;\cap\;
    \mathcal{C}(\mathbb{R}_+; L^2(\Sigma)).
    $

    \item
    Corollary~\ref{flow} also holds for equation~\eqref{Elog}.
\end{enumerate}
\end{remark}

\section{Construction in \texorpdfstring{$H^1(\Sigma)$}{H1(Sigma)}}
Let $ (\Sigma, g) $ be a smooth compact Riemannian surface without boundary, and let $ \lambda > 0 $.
The construction of solutions in $ H^1(\Sigma) $ relies on a compactness argument.
First, we regularize equation~\eqref{nls}.

\begin{equation}\tag{NLS-$\zeta_\epsilon$}\label{Eeps}
\begin{cases}
i \partial_t u_\varepsilon + \Delta u_\varepsilon + i \lambda \, u_\varepsilon \, \zeta(|u_\varepsilon|+1+\varepsilon) = 0,\\
u_\varepsilon(0) = u_0.
\end{cases}
\end{equation}
\subsection{Construction}
For a fixed $\varepsilon > 0$, the equation \eqref{Eeps} has a unique maximal solution \[
u_\varepsilon \in \mathcal{C}([0,T_{\rm max}^{\varepsilon}); H^1(\Sigma)) \cap \mathcal{C}^1([0,T_{\rm max}^{\varepsilon}); H^{-1}(\Sigma)).
\] (see Chapter 3 of \cite{2}),

The estimates below show that $T_{\rm max}^{\varepsilon} = +\infty$ and provide uniform bounds with respect to $\varepsilon$.

\begin{lemma}\label{mass}
    For all $t\in [0,T_{\rm max}^{\varepsilon}[,$ we have $$\norm{u_\varepsilon(t)}_2\leq \norm{u_0}_2e^{-\lambda t}.$$ 
\end{lemma}
\begin{proof}
    Multiplying \eqref{Eeps} by $\overline{u_\varepsilon}$, integrating over $\Sigma$, and taking the real part yields
\[
\frac{1}{2} \frac{d}{dt} \|u_\varepsilon(t)\|_{L^2(\Sigma)}^2
= -\lambda \int_\Sigma |u_\varepsilon(t)|^2 \, \zeta(|u_\varepsilon(t)| + 1 + \varepsilon) \, dx.
\]
Under the assumption $\lambda>0$ and $\zeta(\cdot)\ge 1$, we obtain in particular
\begin{equation}\label{mass_bound}
\partial_t\|u_\varepsilon(t)\|_{L^2(\Sigma)}^2 \le -2\lambda\|u_\varepsilon(t)\|_{L^2(\Sigma)}^2,
\end{equation}
Hence $$\norm{u_\varepsilon(t)}_2\leq e^{-\lambda t}\norm{u_0}_2.$$
Thus, the mass is uniformly bounded in $\varepsilon$.
\end{proof}

\begin{lemma}\label{3}
For all $t\in [0,T_{\rm max}^{\varepsilon}[$, we have
\begin{equation}\label{grad_bound}
\|\nabla u_\varepsilon(t)\|_{L^2(\Sigma)} \le \|\nabla u_0\|_{L^2(\Sigma)}.
\end{equation}
\end{lemma}

\begin{proof}We have
\[
\frac{1}{2} \frac{d}{dt} \|\nabla u_\varepsilon(t)\|_{L^2(\Sigma)}^2
= \Re \langle \partial_t u_\varepsilon(t), -\Delta u_\varepsilon(t) \rangle.
\]
Replacing $\partial_t u_\varepsilon = -i \Delta u_\varepsilon - \lambda u_\varepsilon \zeta(|u_\varepsilon|+1+\varepsilon)$, the term involving $\Delta$ is purely imaginary and vanishes, giving
\[
\frac{1}{2} \frac{d}{dt} \|\nabla u_\varepsilon(t)\|_{L^2(\Sigma)}^2
= -\lambda \, \Re \langle u_\varepsilon(t) \zeta(|u_\varepsilon(t)|+1+\varepsilon), \Delta u_\varepsilon(t) \rangle.
\]
By integration by parts,
\[
\Re \langle u_\varepsilon(t) \zeta(|u_\varepsilon(t)|+1+\varepsilon), \Delta u_\varepsilon(t) \rangle
= - \Re\int_\Sigma \nabla \big( u_\varepsilon(t) \zeta(|u_\varepsilon(t)|+1+\varepsilon) \big) \cdot \overline{\nabla u_\varepsilon(t)} \, dx.
\]
Expanding the gradient and taking the real part, we obtain
\begin{equation}\label{10}
    \frac{1}{2} \frac{d}{dt} \|\nabla u_\varepsilon(t)\|_{L^2(\Sigma)}^2
= -\lambda \int_\Sigma \zeta(|u_\varepsilon(t)|+1+\varepsilon) |\nabla u_\varepsilon(t)|^2
- \lambda \int_\Sigma |u_\varepsilon(t)| \zeta'(|u_\varepsilon(t)|+1+\varepsilon) \Re \frac{u_\varepsilon(t) \nabla |u_\varepsilon(t)| \cdot \overline{\nabla u_\varepsilon(t)}}{|u_\varepsilon(t)|} \, dx.
\end{equation}
By Lemma \ref{cr}, the function $f:x \mapsto x \zeta(x+1+\varepsilon)$ is non-decreasing, so $-f'(x) \le 0$. In particular,
\[
-\zeta(|u_\varepsilon|+1+\varepsilon)\leq|u_\varepsilon| \zeta'(|u_\varepsilon|+1+\varepsilon)\leq 0.
\]
In the other hand we have \[
\left| \Re \frac{u_\varepsilon \nabla |u_\varepsilon| \cdot \overline{\nabla u_\varepsilon}}{|u_\varepsilon|} \right| \le |\nabla u_\varepsilon|^2\quad \text{ and } \quad \zeta'(|u_\varepsilon|+1+\varepsilon) \le 0.
\]
Thus,
\[
\frac{1}{2} \frac{d}{dt} \|\nabla u_\varepsilon(t)\|_{L^2(\Sigma)}^2 
\le \lambda \int_\Sigma |u_\varepsilon(t)| \zeta'(|u_\varepsilon(t)|+1+\varepsilon) \left( |\nabla u_\varepsilon(t)|^2 - \Re \frac{ u_\varepsilon(t)\nabla |u_\varepsilon(t)| \cdot \overline{\nabla u_\varepsilon(t)}}{|u_\varepsilon(t)|} \right) dx \le 0.
\]
By the fundamental theorem of calculus, the lemma follows:
\[
\|\nabla u_\varepsilon(t)\|_{L^2(\Sigma)} \le \|\nabla u_0\|_{L^2(\Sigma)}.
\]
\end{proof}
\begin{remark}
The previos lemmas show that
\begin{equation}\label{nt}
    M:=\sup_{\substack{0<\varepsilon<1 \\ t\in \mathbb{R}_+}}\|u_\varepsilon(t)\|_{H^1(\Sigma)} <\infty,
\end{equation}
which allows the solution to be extended globally in time: $T_{\rm max}^{\varepsilon} = +\infty$.
\end{remark}
We now need to prove that the regularized solutions $u_\varepsilon$ converge, up to the extraction
of a subsequence, to some limit function $u$ belonging to the some functional space.
 
\begin{lemma}\label{ona}
Let $T>0$, then up to extraction of a subsequence, the family $(u_\varepsilon)_\varepsilon$ converges in $\mathcal{C}([0,T]; L^2(\Sigma))$, and the limit is denoted by $u$.
\end{lemma}

\begin{proof}
From the previous estimates, the family $\{ u_\varepsilon \}_{0<\varepsilon<1}$ is uniformly bounded in $L^\infty(\mathbb{R}_+; H^1(\Sigma))$. Moreover, from equation \eqref{Eeps},
\[
\partial_t u_\varepsilon = -i \Delta u_\varepsilon - \lambda u_\varepsilon \zeta(|u_\varepsilon|+1+\varepsilon).
\]
The first term belongs to $L^\infty(\mathbb{R}_+; H^{-1}(\Sigma))$ (since $\Delta u_\varepsilon \in H^{-1}$) and the second term is uniformly bounded in $L^\infty(\mathbb{R}_+; L^2(\Sigma))$. Since $L^2(\Sigma) \hookrightarrow H^{-1}(\Sigma)$, we have
\[
\sup_\varepsilon \|\partial_t u_\varepsilon\|_{L^\infty(0,T; H^{-1}(\Sigma))} < \infty.
\]
We then apply the Aubin–Lions lemma: since
\[
H^1(\Sigma) \overset{\text{compact}}{\hookrightarrow} L^2(\Sigma) \overset{\text{continuous}}{\hookrightarrow} H^{-1}(\Sigma),
\]
and $(u_\varepsilon)$ is bounded in $L^\infty(0,T; H^1(\Sigma))$ while $(\partial_t u_\varepsilon)$ is bounded in $L^\infty(0,T; H^{-1}(\Sigma))$, there exists a subsequence (still denoted $u_\varepsilon$) and a function
\[
u \in \mathcal{C}([0,T]; L^2(\Sigma)),
\]
such that
\[
u_\varepsilon \to u \quad \text{in }\quad \mathcal{C}([0,T]; L^2(\Sigma)).
\]
\end{proof}

\begin{corollary}\label{cv}
The function $u$, defined in Lemma \ref{ona}, satisfies:
\begin{enumerate}
    \item $u \in L^{\infty}(\mathbb{R}_+, H^1(\Sigma))$.
    \item $u_\varepsilon(t) \rightharpoonup u(t)$ weakly in $H^1(\Sigma)$, for all $t \in \mathbb{R}_+$.
\end{enumerate}
\end{corollary}

    
\begin{lemma}\label{cv11}
$$
u_\varepsilon \, \zeta(|u_\varepsilon|+1+\varepsilon) \rightharpoonup u \, \zeta(|u|+1) \quad \text{in} \quad \mathcal{D}'(\mathbb{R}_+^*\times \Sigma).
$$
\end{lemma}

\begin{proof}
Let $T>0$. Then
\[
u_\varepsilon \to u
\quad \text{in } \mathcal{C}([0,T]; L^2(\Sigma)),
\]
which implies that
\[
u_\varepsilon(t,x) \to u(t,x)
\quad \text{for almost every } (t,x) \in (0,T)\times \Sigma.
\]
Therefore, for almost every $ (t,x) \in (0,T) \times \Sigma $ such that
$ u(t,x) \neq 0 $, we have
\[
u_\varepsilon(t,x) \, \zeta(|u_\varepsilon(t,x)| + 1 + \varepsilon)
\longrightarrow
u(t,x) \, \zeta(|u(t,x)| + 1).
\]
Since $(0,T)\times\Sigma$ has finite Lebesgue measure, the convergence holds in measure. On the other hand, we have
\[
\big| u_\varepsilon(t,x)\, \zeta(|u_\varepsilon(t,x)| + 1 + \varepsilon) \big|
\le 1 + |u_\varepsilon(t,x)|
\quad \text{for almost every } (t,x)\in (0,T)\times\Sigma.
\]
Combining this with
\[
\|u_\varepsilon(t)\|_{L^2(\Sigma)} \le \|u_0\|_{L^2(\Sigma)},
\]
we deduce that 
\[
u_\varepsilon(t,x)\, \zeta(|u_\varepsilon(t,x)| + 1 + \varepsilon),
\]
is uniformly integrable in $L^1((0,T)\times\Sigma)$.
By Vitali's theorem, we obtain
\[
u_\varepsilon \, \zeta(|u_\varepsilon| + 1 + \varepsilon)
\longrightarrow
u \, \zeta(|u| + 1)
\quad \text{in } L^1((0,T)\times \Sigma).
\]
and therefore also in $\mathcal D'((0,T)\times \Sigma)$.
\\
Since $T>0$ is arbitrary, this proves that
\[
u_\varepsilon \zeta(|u_\varepsilon| + 1 + \varepsilon) \rightharpoonup u \zeta(|u| + 1) \quad \text{in } \mathcal D'(\mathbb{R}_+^*\times \Sigma).
\]
\end{proof}

\begin{proof}[Proof theorem \ref{1.1}]
Let $\phi \in \mathcal{C}_c^1(\mathbb{R}_+^*)$ and $\psi \in \mathcal{C}_c^\infty(\Sigma)$. Then
\begin{align*}
\int_{\mathbb{R}_+} (i u_\varepsilon(t), \psi)_2 \, \phi'(t) \, dt
= \int_{\mathbb{R}_+} \Big[ - (\nabla u_\varepsilon(t), \nabla \psi)_2
+ \lambda (u_\varepsilon(t) \zeta(|u_\varepsilon(t)|+1+\varepsilon), \psi)_2 \Big] \phi(t) \, dt.
\end{align*}
Passing to the limit and using and Lemma \ref{cv11} and Corollary \ref{cv}, we obtain
\[
\int_{\mathbb{R}_+} (i u(t), \psi)_2 \, \phi'(t) \, dt
= \int_{\mathbb{R}_+} \Big[ - (\nabla u(t), \nabla \psi)_2
+ \lambda (u(t) \zeta(|u(t)|+1), \psi)_2 \Big] \phi(t) \, dt.
\]
\end{proof}

\subsection{Uniqueness}
The uniqueness is a direct consequence of Proposition~\ref{pro} by taking $s = 0$.
In order to prove this proposition, we establish a coercivity property for the function
$z \mapsto z\,\zeta(|z|+1)$.
More precisely, we prove the following lemma.

\begin{lemma}\label{uni}
There exists $c>0$ such that, for all $z,s \in \mathbb{C}^*$,
\[
\Re\big( (f(z)-f(s)) (\overline{z}-\overline{s}) \big) \ge c |z-s|^2.
\]
where $f(z):=z\zeta(\abs{z}+1)$
\end{lemma}

\begin{proof}
We write
\[
f(z)-f(s) = \int_0^1 Df(s+t(z-s)) \cdot (z-s) \, dt,
\]
where
\[
Df(w) \cdot h = \zeta(|w|+1) h + w \, \zeta'(|w|+1) \frac{\Re(\overline{w} h)}{|w|}.
\]
Set $\delta_t = s + t(z-s)$ and $h = z-s$. Then
\[
\Re\big((f(z)-f(s))(\overline{z}-\overline{s})\big)
= \Re \int_0^1 \left( \zeta(|\delta_t|+1) |h|^2 + \zeta'(|\delta_t|+1) \frac{(\Re(\overline{\delta_t} h))^2}{|\delta_t|} \right) dt.
\]
Since $\zeta'(\cdot) < 0$, using $(\Re(\overline{\delta_t} h))^2 \le |\delta_t|^2 |h|^2$, we get
\[
\Re\big((f(z)-f(s))(\overline{z}-\overline{s})\big) \ge \int_0^1 (\zeta(|\delta_t|+1) + |\delta_t| \, \zeta'(|\delta_t|+1)) |h|^2 \, dt.
\]
Thanks to Lemma \ref{cr}, we get
\[
\Re\big((f(z)-f(s))(\overline{z}-\overline{s})\big) \ge c |z-s|^2,
\]
for some $c > 0$.
\end{proof}

\begin{proof}[Proof of Proposition \ref{pro}]
     Let $u,v \in L^\infty(\mathbb{R}_+;H^1(\Sigma))$ be two (distributional) solutions of \eqref{nls} with initial data $u_0,v_0$ given by previos theorem. then for all $t\geq s$.

 Set $w = u - v$. Since $u$ and $v$ are solutions, we have
\[
i \partial_t w + \Delta w + i \lambda \big(u \zeta(|u|+1) - v \zeta(|v|+1)\big) = 0.
\]
We take the $L^2(\Sigma)$ inner product with $w$ and consider the real part:
\[
\frac{1}{2} \frac{d}{dt} \|w(t)\|_{L^2(\Sigma)}^2
= \Re \langle \partial_t w(t), w(t) \rangle
= \Re \big\langle -i \Delta w(t) - \lambda (u(t) \zeta(|u(t)|+1) - v(t) \zeta(|v(t)|+1)), w(t) \big\rangle.
\]
Hence,
\[
\frac{d}{dt} \|w(t)\|_{L^2(\Sigma)}^2
= -2 \lambda \, \Re \langle u(t) \zeta(|u(t)|+1) - v(t) \zeta(|v(t)|+1), w(t) \rangle.
\]
Thanks to the Lemma \ref{uni} thus $$\frac{d}{dt} \|w(t)\|_{L^2(\Sigma)}^2\leq -2\lambda c\norm{w(t)}_{L^2(\Sigma)}^2.$$ so  that $$\norm{w(t)}_{L^2(\Sigma)}\leq e^{-\lambda c(t-s)}\norm{w(s)}_{L^2(\Sigma)}.$$
\end{proof}

\begin{proof}[Proof of Corollary \ref{flow}]
Let $ T > 0 $ and let $ u_{n,0},\, u_0 \in H^1(\Sigma) $. Denote by $ u_n(t) $ and $ u(t) $ the solutions to \eqref{nls} corresponding to the initial data $ u_{n,0} $ and $ u_0 $, respectively.
Assume that
\[
u_{n,0} \longrightarrow u_0
\quad \text{in } L^2(\Sigma),
\]
and that there exists a constant $ K > 0 $ such that
\begin{equation}\label{000}
\sup_{\substack{n\geq 0 \\ t\in[0,T]}} \|u_n(t)\|_{H^1(\Sigma)} \le K.
\end{equation}
Applying Proposition~\ref{pro} with $ s = 0 $, we obtain
\[
u_n(t) \longrightarrow u(t)
\quad \text{strongly in } L^2(\Sigma),
\quad \text{for all } t \in [0,T].
\]
Moreover, thanks to the uniform bound \eqref{000} and the uniqueness of the limit, it follows that
\[
u_n(t) \rightharpoonup u(t)
\quad \text{weakly in } H^1(\Sigma),
\quad \text{for all } t \in [0,T].
\]
\end{proof}
\subsection{Finite time extinction}
In this subsection, we assume that $d=1$ and we prove that the solution of
equation~\eqref{nls}, defined in Theorem~\ref{1.1}, vanishes after a finite time
$T^*<\infty$. The main idea of the proof follows exactly the strategy developed
in~\cite{3} (see also \cite{10}), based on the Nash--Moser inequality. For this purpose, we recall
below some standard inequalities inspired by that approach.
\begin{lemma}\label{zero}
Let $y \in \mathcal{C}^1([0,+\infty))$ be a  positive function and let
$\alpha \in (0,1)$. 
Assume that there exists a constant $k>0$ such that
\[
y'(t) + k\, y(t)^{\alpha} \le 0, \qquad \forall t \ge 0.
\]
Then there exists a finite time $T^*\geq0$ such that
\[
y(T^*)=0.
\]
In particular $$y(t)=0 \qquad t \geq T^*$$
\end{lemma}

\begin{proof}
Suppose, by contradiction, that $y(t) \neq 0$ for all $t \ge 0$.
Since
\begin{equation}\label{se}
    y'(t) \le -k\, y(t)^{\alpha},
\end{equation}
for $y(t) > 0$, dividing both sides by $y(t)^\alpha$ gives
\[
\frac{y'(t)}{y(t)^\alpha} \le -k.
\]
Moreover,
\[
\frac{d}{dt} \bigl( y(t)^{1-\alpha} \bigr) 
= (1-\alpha) \frac{y'(t)}{y(t)^\alpha},
\]
so that
\[
\frac{d}{dt} \bigl( y(t)^{1-\alpha} \bigr) \le -(1-\alpha) k.
\]
Integrating over $[0,t]$, we obtain
\[
y(t)^{1-\alpha} \le y(0)^{1-\alpha} - (1-\alpha) k t.
\]
Letting $t \to +\infty$, the right-hand side tends to $-\infty$, which contradicts $y(t) > 0$.  
\\
Hence, there exists $T^* > 0$ such that
\[
y(T^*) = 0.
\]
Now, by the inequality \eqref{se}, $y(t)$ is non-increasing, so
\[
y(t) \le y(T^*) = 0 \qquad \forall t \ge T^*.
\]
Therefore,
\[
y(t) = 0 \qquad \forall t \ge T^*.
\]

\end{proof}

The crucial observation is that
\begin{equation}\label{s}
\frac{d}{dt}\|u(t)\|_{L^2}^2 \le -\lambda \|u(t)\|_{L^1}.
\end{equation}
Hence, to apply the previous lemma, one needs an estimate of the $L^2$-norm
in terms of the $L^1$-norm.
\begin{lemma}[\cite{3} Lemma 4.1]
    Let $(\Sigma,g)$ be a smooth compact Riemannian manifold (without boundary) of dimension one. 
There exists a constant $C>0$ such that $$\norm{v}_2^3\leq C \norm{v}_1^2\norm{v}_{H^1} \quad \text{ for all } \quad v\in H^1.$$
\end{lemma}
We now prove the Finite time extinction of the solution. 
Thanks to the estimate~\eqref{s} and the previous inequality, we obtain
\[
\partial_t \|u(t)\|_{L^2}^2
\le
-\frac{k}{\|u_0\|_{H^1}^{1/2}}\,\|u(t)\|_{L^2}^{3/2}.
\]
For some $k>0$, by Lemma~\ref{zero}, this implies that there exists a time
$T^*>0$, depending only on $\|u_0\|_{H^1}$, such that
\[
u(t, x)=0,\quad \text{for almost every } x \in \Sigma, \quad \text{for all } t \ge T^*.
\]
\section{NLS-\texorpdfstring{$\zeta$} with logarithmic perturbation}{zeta}
In this section, we assume that $ d = 1 $ and consider the NLS--$\zeta$ equation perturbed by a logarithmic term:
\begin{equation}\tag{logNLS-$\zeta$}
\begin{cases}
i \partial_t u + \Delta u
+ i \lambda\, u\, \zeta(|u| + 1)
+ \mu\, u\, \log(|u|^2) = 0,
& (t,x) \in \mathbb{R}_+ \times \Sigma, \\[0.3em]
u(0,x) = u_0(x) \in H^1(\Sigma),
\end{cases}
\end{equation}
where $ \lambda > 0 $ and $ \mu \in \mathbb{R} $.

As in the previous section, we consider the following regularized problem:
\begin{equation}\tag{logNLS-$\zeta_\varepsilon$}\label{Elogeps}
\begin{cases}
i \partial_t u_\varepsilon + \Delta u_\varepsilon
+ i \lambda\, u_\varepsilon\, \zeta(|u_\varepsilon| + 1 + \varepsilon)
+ \mu\, u_\varepsilon \log(|u_\varepsilon|^2 + \varepsilon) = 0, \\[0.3em]
u_\varepsilon(0,x) = u_0(x).
\end{cases}
\end{equation}

\subsection{Construction in \texorpdfstring{$H^1(\Sigma)$}{H1(Sigma)}}
Again for a fixed $\varepsilon > 0$, the equation \eqref{Elogeps} has a unique maximal solution \[
u_\varepsilon \in \mathcal{C}([0,T_{\rm max}^{\varepsilon}); H^1(\Sigma)) \cap \mathcal{C}^1([0,T_{\rm max}^{\varepsilon}); H^{-1}(\Sigma)).
\] (see Chapter 3 of \cite{2}),
\begin{lemma}
    For all $t\in [0,T_{\rm max}^{\varepsilon}[$, we have
    $$\norm{u_\varepsilon(t)}_2\leq \norm{u_0}e^{-\lambda t} \qquad \text{ and } \qquad \norm{\nabla u_\varepsilon(t)}_2\leq e^{\abs{\mu}t}\norm{\nabla u_0}_2$$
\end{lemma}
\begin{proof}
    For the first estimate Multiplying \eqref{Elogeps} by $\overline{u_\varepsilon}$ and taking the real part gives
\[
\frac{1}{2}\frac{d}{dt}\|u_\varepsilon(t)\|_{L^2(\Sigma)}^2
= -\lambda \int_\Sigma |u_\varepsilon(t)|^2 \zeta(|u_\varepsilon(t)|+1+\varepsilon)\,dx,
\]
so that, as before,
\begin{equation}\label{L2boundlog}
\|u_\varepsilon(t)\|_{L^2(\Sigma)} \le \|u_0\|_{L^2(\Sigma)}, \quad \forall t\ge0.
\end{equation}
For the second estimate, we compute
\[
\frac{1}{2}\frac{d}{dt}\|\nabla u_\varepsilon(t)\|_{L^2(\Sigma)}^2
= -\lambda\,\Re \int_\Sigma \nabla(u_\varepsilon(t)\zeta(|u_\varepsilon(t)|+1+\varepsilon)) \cdot \overline{\nabla u_\varepsilon(t)}
    - \mu\,\Im \int_\Sigma \nabla(u_\varepsilon(t)\log(|u_\varepsilon(t)|^2+\varepsilon)) \cdot \overline{\nabla u_\varepsilon(t)}.
\]
The first term is handled as before and is nonpositive.  
The second term can be estimated by
\[
\left|\Im \int_\Sigma \nabla(u_\varepsilon(t)\log(|u_\varepsilon(t)|^2+\varepsilon)) \cdot \overline{\nabla u_\varepsilon(t)}\right|
\le \int_\Sigma |\nabla u_\varepsilon(t)|^2\,dx.
\]
Applying Grönwall’s inequality, we obtain
\begin{equation}\label{gradboundlog}
\|\nabla u_\varepsilon(t)\|_{L^2(\Sigma)} \le e^{|\mu| t} \|\nabla u_0\|_{L^2(\Sigma)}.
\end{equation}
\end{proof}

The previos lemma proof that \[
C_T:=\sup_{\substack{0<\varepsilon<1 \\ t\in[0,T]}}\|u_\varepsilon(t)\|_{H^1(\Sigma)} < \infty,
\]
so that the family $(u_\varepsilon)_{0<\varepsilon<1}$ is uniformly bounded in $L^\infty(0,T; H^1(\Sigma))$.  
Repeating the compactness argument as in the \eqref{nls} case we get the proposition.
\begin{proposition}\label{of}
Let $T>0$, then up to extraction of a subsequence, the family $(u_\varepsilon)_\varepsilon$ converges in $\mathcal{C}([0,T]; L^2(\Sigma))$, and the limit is denoted by $u$. 

In the addition  
\begin{enumerate}
    \item $u \in L^{\infty}_{\rm loc}(\mathbb{R}_+, H^1(\Sigma))$.
    \item $u_\varepsilon(t) \rightharpoonup u(t)$ weakly in $H^1(\Sigma)$, for all $t \in \mathbb{R}_+$.
\end{enumerate}
\end{proposition}
We now aim to prove that the function defined in the previous proposition
belongs to $L^\infty(\mathbb{R}_+; H^1)$. for this we propose that to proof the finit time vanisching we have \begin{lemma}
    There exist $T^*>0$ such that $$u(t)=0 \quad \text{ in } L^2(\Sigma), \quad \text{ for all } t\geq T^*.$$
    In particular $u\in L^{\infty}(\mathbb{R}_+,H^1(\Sigma)).$
\end{lemma}
\begin{proof}
    We simply observe that $u$ satisfies the inequality \eqref{s}, which in turn implies the Finite time extinction property. The second is its proof, obtained by combining the previous proposition with the Finite time extinction property.
\end{proof}

\subsection{Uniqueness}
Let's beginig to proof the uniqnesse for the equation \eqref{Elog} for this we want recall inequality for logarithime 
\begin{lemma}[\cite{6} Lemme 1.1.1]\label{log}
    Let $z_1,z_2\in \mathbb{C}$ then \[
\big|\Im( z_1\log(|z_1|)-z_2\log(|z_2|)( \overline{z_1}-\overline{z_2}))\big| \le |z_1-z_2|^2,
\]
\end{lemma}
\begin{proposition}[Uniqueness]
Let $u_1, u_2 \in L^{\infty}(0,T; H^1(\Sigma))$ be two solutions of \eqref{Elog} (in the distributional sense). If $\lambda > 0$, then $u_1 = u_2$.
\end{proposition}

\begin{proof}
Set $w = u_1 - u_2$. Then
\[
i\partial_t w + \Delta w + i\lambda\big(u_1\zeta(|u_1|+1) - u_2\zeta(|u_2|+1)\big)
+ 2\mu\big(u_1\log(|u_1|) - u_2\log(|u_2|)\big) = 0.
\]
Taking the $L^2$ inner product with $w$ and the real part, we get
\begin{align*}
    \frac{1}{2}\frac{d}{dt}\|w(t)\|_{L^2(\Sigma)}^2
&= -\lambda\,\Re\langle u_1(t)\zeta(|u_1(t)|+1)-u_2(t)\zeta(|u_2(t)|+1), w(t)\rangle
    \\&\qquad -2\mu\,\Im\langle u_1(t)\log(|u_1(t)|)-u_2(t)\log(|u_2(t)|), w(t)\rangle.
\end{align*}
Thanks to Lemma \ref{uni} the first term is nonpositive, the second term can be estimated by Lemma \ref{log}
so by Grönwall’s lemma,
\[
\|w(t)\|_{L^2(\Sigma)} \le e^{|\mu| t} \|w(0)\|_{L^2(\Sigma)}.
\]
If $u_1(0)=u_2(0)$, it follows that $w\equiv 0$.
\end{proof}\\

\renewcommand{\appendixpagename}{Appendix}
\renewcommand{\appendixtocname}{Appendix}

\appendix
\appendixpage
\addappheadtotoc
\setcounter{theorem}{0}

\section{Estimates for the Riemann Zeta Function}\label{A}

\begin{lemma}
For all $x > 1$, we have
\[
\frac{1}{x - 1} \leq \zeta(x) \leq \frac{1}{x - 1} + 1.
\]
\end{lemma}

\begin{proof}
Recall that for any continuous decreasing function $f$, we have
\[
f(n) \leq \int_{n-1}^{n} f(x) \, dx \leq f(n-1).
\]
Applying this to $f(x) = 1/x^s$ for $s>1$, which is decreasing, gives the desired bounds.
\end{proof}
\begin{lemma}\label{mo}
    Let $f$ be a strictly decreasing function integrable on $[0,1]$. Then
\[
\int_0^1 \left(x - \frac{1}{2}\right) f(x)\,dx < 0.
\]
\end{lemma}
\begin{proof}
    We have \begin{align*}
        \int_0^1\left(x - \frac{1}{2}\right) f(x)\,dx&=\int_{0}^{1/2}\left(x - \frac{1}{2}\right) f(x)\,dx+\int_{1/2}^1\left(x - \frac{1}{2}\right) f(x)\,dx\\
        &=-\int_{0}^{1/2}\left( \frac{1}{2}-x\right) f(x)\,dx+\int_{0}^{1/2}\left(\frac{1}{2}-x\right) f(1-x)\,dx\\&=\int_{0}^{1/2}\left( \frac{1}{2}-x\right) \left(f(1-x)-f(x)\right)dx<0.
    \end{align*}
\end{proof}
\begin{lemma}\label{cr}
For all $\varepsilon \ge 0$, the function
\[
x \mapsto x \zeta(x+1+\varepsilon),
\]
is increasing on $\mathbb{R}_{+}^*$, Moreover $$\zeta(x+1)+x\zeta'(x+1)\geq c,$$
for some $c>0.$
\end{lemma}
\begin{proof}
First, we observe that
\[
x\, \zeta(x + 1 + \varepsilon) = (x + \varepsilon)\, \zeta(x + \varepsilon + 1) - \varepsilon \, \zeta(x + \varepsilon + 1).
\]
This shows that it suffices to prove that the function
\[
x \longmapsto x \, \zeta(x + 1),
\]
is increasing.  

To this end, recall that for all $x > 0$ we have the integral representation
\begin{equation} \label{eq:zeta_integral}
\zeta(x + 1) = 1 + \frac{1}{x} - (x + 1) \int_1^{+\infty} \frac{\{y\}}{y^{x+2}} \, dy,
\end{equation}
where $\{y\} := y - \lfloor y \rfloor$ denotes the fractional part of $y$. From this expression, one can deduce
\[
\zeta(x + 1) = \frac{1}{x} + \frac{1}{2} - (x + 1) \int_1^{+\infty} \frac{\{y\} - \frac{1}{2}}{y^{x+2}} \, dy=\frac{1}{x} + \frac{1}{2} - (x + 1) \sum_{n=1}^{\infty}\int_0^{1} \frac{y - \frac{1}{2}}{(y+n)^{x+2}} \, dy.
\]

Moreover, for $n\geq 1$ the function $y \mapsto (y+n)^{-(x+2)}$ is strictly decreasing on $(0, +\infty)$, by Lemma \ref{mo} which implies
\[
\int_0^1 \frac{y - \frac{1}{2}}{(y+n)^{x+2}} \, dy < 0.
\]
Hence, we obtain the inequality
\begin{equation} \label{eq:zeta_lower_bound}
\zeta(x + 1) > \frac{1}{x} + \frac{1}{2}.
\end{equation}
Differentiating \eqref{eq:zeta_integral} with respect to $x$, we obtain
\[
\zeta'(x + 1) = - \frac{1}{x^2} - \int_1^{+\infty} \frac{(\{y\} - \frac{1}{2}) \left(1 - (x + 1) \ln y \right)}{y^{x+2}} \, dy.
\]
It follows that
\[
\left| \zeta'(x + 1) + \frac{1}{x^2} \right|
\le \frac{1}{2} \int_1^{+\infty}
\frac{\lvert 1 - (x + 1)\ln y \rvert}{y^{x+2}} \, dy
= \frac{1}{e(x+1)}.
\]
Indeed,
\begin{align*}
\int_1^{+\infty} \frac{|1 - (x + 1) \ln y|}{y^{x+2}} \, dy
&= \int_1^{e^{\frac{1}{x+1}}}
\frac{1 - (x + 1) \ln y}{y^{x+2}} \, dy
+ \int_{e^{\frac{1}{x+1}}}^{\infty}
\frac{(x + 1) \ln y - 1}{y^{x+2}} \, dy \\
&= -\left[\frac{1 - (x+1)\ln y}{(x+1)y^{x+1}}\right]_1^{e^{\frac{1}{x+1}}}
- \int_1^{e^{\frac{1}{x+1}}} y^{-x-2} \, dy \\
&\quad
+ \left[\frac{1 - (x+1)\ln y}{(x+1)y^{x+1}}\right]_{e^{\frac{1}{x+1}}}^{\infty}
+ \int_{e^{\frac{1}{x+1}}}^{\infty} y^{-x-2} \, dy \\
&= \frac{1}{e(x+1)}.
\end{align*}
Combining the inequalities \eqref{eq:zeta_lower_bound} and the above, we obtain
\[
- \frac{\zeta'(x + 1)}{\zeta(x + 1)} < \frac{\frac{1}{x^2} + \frac{1}{e(x+1)}}{\frac{1}{x} + \frac{1}{2}} < \frac{1}{x}.
\]
Thus,
\[
\zeta(x + 1) + x \, \zeta'(x + 1) > 0.
\]
Furthermore, it is well known that as $x \to 0^+$,
\[
\zeta(x + 1) = \frac{1}{x} + \gamma + O(x),
\]
where $\gamma$ denotes the Euler–Mascheroni constant. Hence,
\[
\lim_{x \to 0^+} \big( \zeta(x + 1) + x \, \zeta'(x + 1) \big) = \gamma > 0.
\]
Moreover,
\[
\lim_{x \to +\infty} \big( \zeta(x + 1) + x \, \zeta'(x + 1) \big) = 1.
\]
By continuity of the function
\[
x \longmapsto \zeta(x + 1) + x \, \zeta'(x + 1),
\]
there exists a constant $c > 0$ such that
\[
\zeta(x + 1) + x \, \zeta'(x + 1) \ge c, \quad \forall x > 0.
\]
\end{proof}


\begin{thebibliography}{99}
\addcontentsline{toc}{chapter}{Bibliographie}

\bibitem[Br05]{8}
Kevin A.~Broughan.
\newblock The holomorphic flow of Riemann's function $\xi(z)$.
\newblock {\em Nonlinearity}, 18(3):1269--1294, 2005.
\newblock doi:10.1088/0951-7715/18/3/017.

\bibitem[BrBa04]{9}
Kevin A.~Broughan and A.~Ross Barnett.
\newblock The holomorphic flow of the Riemann zeta function.
\newblock {\em Math. Comput.}, 73(246):987--1004, 2004.
\newblock doi:10.1090/S0025-5718-03-01529-1.

\bibitem[Ca03]{2}
Thierry Cazenave.
\newblock {\em Semilinear Schrödinger Equations}.
\newblock Providence, RI: American Mathematical Society; New York, N.Y.: Courant Institute of Mathematical Sciences, New York University, 2003.

\bibitem[CaGa11]{3}
R\'emi Carles and Cl\'ement Gallo.
\newblock Finite time extinction by nonlinear damping for the Schr\"odinger equation.
\newblock {\em Commun. Partial Differ. Equations}, 36(4--6):961--975, 2011.
\newblock doi:10.1080/03605302.2010.531074.

\bibitem[CaHa80]{6}
T.~Cazenave and A.~Haraux.
\newblock {\em Équations d'évolution avec non linéarité logarithmique}.
\newblock Annales de la Faculté des Sciences de Toulouse, 5e série, II:21--55, 1980.

\bibitem[CaMuPoSa25]{7}
V\'ictor Castillo, Claudio Mu\~noz, Felipe Poblete, and Vicente Salinas.
\newblock The generalized Riemann zeta heat flow.
\newblock {\em J. Funct. Anal.}, 288(10):44, 2025.
\newblock doi:10.1016/j.jfa.2025.110879.


\bibitem[Ca08]{10}
Alexandre Cabot.
\newblock Stabilization of oscillators subject to dry friction: finite time convergence versus exponential decay results.
\newblock {\em Trans. Am. Math. Soc.}, 360(1):103--121, 2008.
\newblock doi:10.1090/S0002-9947-07-03990-6.

\bibitem[MoVa07]{5}
H.~L. Montgomery and R.~C. Vaughan.
\newblock {\em Multiplicative Number Theory I: Classical Theory}.
\newblock Cambridge Studies in Advanced Mathematics, Vol.~97,
Cambridge University Press, Cambridge, 2007.

\bibitem[Ni59]{4}
L. Nirenberg.
\newblock On elliptic partial differential equations.
\newblock {\em Ann. Scuola Norm. Sup. Pisa (3)}, 13:115--162, 1959.

\bibitem[HaOz25]{1}
Masayuki Hayashi and Tohru Ozawa.
\newblock The Cauchy problem for the logarithmic Schr\"odinger equation revisited.
\newblock {\em Ann. Henri Poincar\'e}, 26(4):1209--1238, 2025.
\newblock doi:10.1007/s00023-024-01460-z.















\end{thebibliography}
\end{document}